\documentclass[12pt]{amsart}
\usepackage{amsmath,amsthm,amssymb,fullpage,caption}
\usepackage{tikz}
\usetikzlibrary{calc}
\usetikzlibrary{shapes.symbols}
\usetikzlibrary{arrows}
\usetikzlibrary{decorations.markings,decorations.pathmorphing}
\usetikzlibrary{shapes}
\usetikzlibrary{patterns}
\newcommand{\size}[1]{\left \vert #1 \right \vert}
\newcommand{\ceil}[1]{\left \lceil #1 \right \rceil}
\newcommand{\floor}[1]{\left \lfloor #1 \right \rfloor}
\newcommand{\cart}{\, \Box \,}

\newcommand{\loc}[1]{\zeta(#1)}

\newcommand{\vl}{v_{L}}
\newcommand{\vr}{v_R}
\newcommand{\dl}{d_L}
\newcommand{\dr}{d_R}

\newtheorem{theorem}{Theorem}[section]

\newtheorem{cor}[theorem]{Corollary}

\begin{document}

\title{Bounds on the localization number}

\author{Anthony Bonato}
\address{Department of Mathematics, Ryerson University, Toronto, ON, Canada, M5B 2K3}
\email{\tt abonato@ryerson.ca}

\author{William B. Kinnersley}
\address{Department of Mathematics, University of Rhode Island, University of Rhode Island, Kingston, RI,USA, 02881}
\email{\tt billk@uri.edu}


\begin{abstract}
We consider the localization game played on graphs, wherein a set of cops attempt to determine the exact location of an invisible robber by exploiting distance probes. The corresponding optimization parameter for a graph $G$ is called the localization number and is written $\loc{G}$. We settle a conjecture of \cite{nisse1} by providing an upper bound on the chromatic number as a function of the localization number. In particular, we show that every graph with $\loc{G} \le k$ has degeneracy less than $3^k$ and, consequently, satisfies $\chi(G) \le 3^{\loc{G}}$.  We show further that this degeneracy bound is tight. We also prove that the localization number is at most 2 in outerplanar graphs, and we determine, up to an additive constant, the localization number of hypercubes.
\end{abstract}

\maketitle

\begin{section}{Introduction}
\emph{Graph searching} focuses on the analysis of games and graph processes that model some form of intrusion in a network and efforts to eliminate or contain that intrusion. One of the best known examples of graph searching is the game of {\em Cops and Robbers}, wherein a robber is loose on the network and a set of cops attempts to capture the robber. How the players move and the rules of capture depend on which variant is studied. There are many variants of graph searching studied in the literature, which are either motivated by problems in practice or inspired by foundational issues in computer science, discrete mathematics, and artificial intelligence, such as robotics and network security. For a survey of graph searching see \cite{bp,by,fomin}, and see \cite{bonato} for more background on Cops and Robbers.

We focus in the present paper on a variant of Cops and Robbers, called the {\em localization game}, in which the cops only have partial information on the location of the robber. The variant we discuss is motivated by a real-world tracking problem with mobile receivers and a cell phone user. The receivers are placed in various locations, and the user is in motion and is only detectable by the strength of their signal to the receivers (measured by their distance to the receivers). The receivers, who do not know the user's location, may appear anywhere and relocate over time. The goal is to uniquely determine the location of the user. See, for example, \cite{bahl}.

The localization game was first introduced for one receiver by Seager \cite{seager1,seager2} and was further studied in \cite{brandt,car}. In this game, there are two players moving on a connected graph, with one player controlling a set of $k$ \emph{cops}, where $k$ is a positive integer, and the second controlling a single \emph{robber}. Unlike in Cops and Robbers, the cops play with imperfect information: the robber is invisible to the cops during gameplay. The game is played over a sequence of discrete time-steps; a \emph{round} of the game is a move by the cops together with the subsequent move by the robber. The robber occupies a vertex of the graph, and when the robber is ready to move during a round, he may move to a neighboring vertex or remain on his current vertex. A move for the cops is a placement of cops on a set of vertices (note that the cops are not limited to moving to neighboring vertices). At the beginning of the game, the robber chooses his starting vertex. After this, the cops move first, followed by the robber; thereafter, the players move on alternate steps. Observe that any subset of cops may move in a given round. In each round, the cops occupy a set of vertices $u_1, u_2, \ldots , u_k$ and each cop sends out a \emph{cop probe}, which gives their distance $d_i$, where $1\le i \le k$, from $u_i$ to the robber. Hence, in each round, the cops determine a \emph{distance vector} $(d_1, d_2, \ldots , d_k)$ of cop probes, which is unique up to the ordering of the cops. Note that relative to the cops' position, there may be more than one vertex $x$ with the same distance vector. We refer to such a vertex $x$ as a \emph{candidate}. For example, in an $n$-vertex clique with a single cop, so long as the cop is not on the robber's vertex, there are $n-1$ many candidates. The cops win if they have a strategy to determine, after finitely many rounds, a unique candidate, at which time we say that the cops {\em capture} the robber. If there is no unique candidate in a given round, then the robber may move in the next round and the cops may move to other vertices resulting in an updated distance vector. The robber wins if he is never captured.

For a connected graph $G$, define the \emph{localization number} of $G$, written $\loc{G}$, to be the least integer $k$ for which $k$ cops have a winning strategy over any possible strategy of the robber (that is, we consider the worst case that the robber a priori knows the entire strategy of the cops). As placing a cop on each vertex gives a distance vector with unique value of $0$ on the location of the robber, $\loc{G}$ is at most $n$ and hence is well-defined.

The localization number is related to the metric dimension of a graph, in a way that is analogous to how the cop number is related to the domination number. The \emph{metric dimension} of a graph $G$, written $\mathrm{dim}(G)$, is the minimum number of cops needed in the localization game so that the cops can win in one round; see \cite{hm,slater}. Hence, $\loc{G} \le \mathrm{dim}(G)$, but in many cases this inequality is far from tight. The bound of $\loc{G} \le \floor{\frac{(\Delta +1)^2}{4}} +1$, where $\Delta$ is the maximum degree of $G$, was shown in \cite{has}. In \cite{nisse1}, Bosek et al. showed that $\loc{G}$ is bounded above by the pathwidth of $G$ and that the localization number is unbounded even on graphs obtained by adding a universal vertex to a tree. They also proved that computing $\loc{G}$ is \textbf{NP}-hard for graphs with diameter 2, and they studied the localization game for geometric graphs. The \emph{centroidal localization game} was considered in \cite{nisse2}, where it was proved, among other things, that the centroidal localization number (and hence the localization number) of outerplanar graphs is at most 3.  In \cite{dfp}, the localization number was studied for binomial random graphs with diameter 2.

Bosek et al. conjectured (see \cite{nisse1}, Conjecture 16) that there is a function $f$ such that every graph with $\loc{G} \le k$ satisfies $\chi(G) \le f(k),$ where $\chi(G)$ is the chromatic number of $G$. We settle this conjecture in Corollary~\ref{cor:chromatic}. In particular, by exploiting a lower bound on the localization number using graph degeneracy, we show that $\chi(G) \le 3^{\loc{G}}$. The degeneracy bound is proven to be tight via a non-trivial example utilizing a graph built from strong powers of cycles. In Theorem~\ref{thm:outerplanar}, we prove that outerplanar graphs have localization number at most 2. We finish by giving an asymptotically tight upper bound on the localization number of the hypercube; in particular, in Theorem~\ref{thm:hypercube}, we show that for all positive integers $n$, $\loc{Q_n} \le \ceil{\log_2 (n-1)} + 2$.

Throughout, all graphs considered are simple, undirected, connected, and finite. For a reference on graph theory, see \cite{Wes01}.
\end{section}

\begin{section}{Degeneracy and localization}

Our first result is a general lower bound on the localization number of a graph in terms of its degeneracy.  The {\em degeneracy} of a graph $G$ is the maximum, over all subgraphs $H$ of $G$, of $\delta(H)$. Note that the degeneracy of any nonempty graph must be a positive integer. For a vertex $u$ in a graph $G$, we define $N_G[u]$ to be the set of neighbors of $u$ along with the vertex $u$ itself.

% degeneracy bound
\begin{theorem}\label{thm:degeneracy}
If $G$ is a graph with degeneracy $k$, where $k$ is a positive integer, then $\loc{G} \ge \log_3 (k+1)$.
\end{theorem}
\begin{proof}
Let $G$ be a graph with degeneracy $k$ and let $H$ be a subgraph of $G$ with $\delta(H) = k$.  Suppose we play the localization game on $G$ with $m$ cops.  It suffices to show that the robber can win provided that $m < \log_3 (k+1)$.  In particular, we show how he can perpetually evade capture while always occupying a vertex of $H$.

Toward this end, we claim that for all $v \in V(H)$, and for every cop probe $(u_1, u_2, \dots, u_m)$, there are at least two vertices in $N_H[v]$ sharing the same distance vector.  Let $d_i = d_G(u_i,v)$, and note that for all $w \in N_H[v]$ we have $d_G(u_i,w) \in \{d_i-1,d_i,d_i+1\}$.  Thus, between them, the vertices of $N_H[v]$ correspond to at most $3^m$ different distance vectors.  Since $m < \log_3 (k+1)$, there are at most $k$ distance vectors represented in $N_H[v]$; since $\size{N_H[v]} \ge k+1$, by the Pigeonhole Principle some distance vector corresponds to at least two vertices in $N_H[v]$, as claimed.

The robber's strategy is now straightforward.  Suppose that, on some robber turn, the robber occupies some vertex $v$ in $H$.  If in fact the robber is choosing an initial position, then he instead pretends that he already occupies some arbitrary vertex $v$ of $H$ and wishes to move to some neighbor of $v$.  Before making his move, the robber considers the cops' subsequent probe.  He next finds some two vertices in $N_H[v]$, say $w$ and $x$, that share the same distance vector with respect to this probe.  The robber moves to $w$; the cops cannot uniquely locate him, since to the best of their knowledge, he could occupy either $w$ or $x$.  Thus the game continues.  The robber can repeat this strategy indefinitely, thereby forever evading capture.
\end{proof}

Johnson and Koch \cite{joh} proved that under a slightly different model of the localization game, if $\loc{G} = 1$, then $\chi(G) \le 4$.  In the game they studied, the robber was not allowed to move to a vertex that the cops had just probed.  Our model gives the robber slightly more power and thus can slightly lower the localization number.  In particular, under our model, if $\loc{G} = 1$, then $\chi(G) \le 3$.  Bosek et al.\ \cite{nisse1} asked whether $\chi(G)$ is, in general, bounded above by some function of $\loc{G}$.  We answer this question in the affirmative; Theorem~\ref{thm:degeneracy} yields a short proof.

% chromatic number corollary
\begin{cor}\label{cor:chromatic}
For every graph $G$, we have $\chi(G) \le 3^{\loc{G}}$.
\end{cor}
\begin{proof}
Let $G$ be any graph and let $k$ be its degeneracy.  It is well-known that $\chi(G) \le k+1$, which in turn is at most $3^{\loc{G}}$ by Theorem~\ref{thm:degeneracy}.
\end{proof}

When $G$ is bipartite, Theorem~\ref{thm:degeneracy} can be improved.

\begin{theorem}\label{thm:degeneracy_bipartite}
If $G$ is a bipartite graph with degeneracy $k$, where $k$ is a positive integer, then $\loc{G} \ge \log_2 k$.
\end{theorem}
\begin{proof}
The proof proceeds exactly as with Theorem~\ref{thm:degeneracy}, except that for all $w \in N_H(v)$ we now have $d_G(u_i,w) \in \{d_i-1,d_i+1\}$, since no neighbor of $v$ occupies the same partite set as $v$.  Thus the vertices of $N_H(v)$ correspond to at most $2^m$ different distance vectors, so if $m < \log_2 k$, then some distance vector corresponds to more than one vertex in $N_H(v)$.
\end{proof}

We remark that results analogous to Theorem \ref{thm:degeneracy} and Corollary \ref{cor:chromatic} are known for metric dimension.  Chartrand et al.\ \cite{chart} showed that $\mathrm{dim(G)} \ge \log_3(\Delta(G)+1)$, while Chappell et al.\ \cite{chap} showed that if $\mathrm{dim}(G) = m$, then $\chi(G) \le 2^m$; both bounds were shown to be tight.

% tightness of bound
We conclude this section by showing that Theorem~\ref{thm:degeneracy} is tight.  To do this we produce, for all $k$, a graph $G_k$ with degeneracy $k$ and localization number $\log_3 (k+1)$.  Recall that the {\em strong product} of graphs $G$ and $H$ is the graph with vertex set $V(G) \times V(H)$, where $(u,v)$ is adjacent to $(u',v')$ provided that $u$ is adjacent to $u'$ in $G$ and $v=v'$, $u = u'$ and $v$ is adjacent to $v'$ in $H$, or $u$ is adjacent to $u'$ in $G$ and $v$ is adjacent to $v'$ in $H$.  We construct $G_k$ as follows.  Begin with the $k$-fold strong product of copies of $C_{40}$.  We refer to the vertices of this strong product as {\em core vertices}, and we represent each one using a $k$-dimensional vector with entries in $\{0,1,\dots, 39\}$; distinct vertices are adjacent provided that they differ by at most 1 (modulo $40$) in every coordinate.

In addition to the core vertices, $G_k$ contains $2k$ {\em satellite vertices}.  For all $i \in \{1,2, \dots, k\}$ and $t \in \{0,10\}$, we add edges joining the satellite vertex $s_{i,t}$ to all core vertices whose $i$th coordinate equals $t$.  We then subdivide each of these edges into a path of length $40$; we refer to the paths produced from this subdivision (including the original endpoints of the edge, namely the satellite and core vertex) as {\em threads} emanating from the corresponding satellite.  We will make repeated use of the following fact: for a core vertex $w$, if $w = (w_1, w_2, \dots, w_{k})$, then $d(s_{i,t},w) = 40 + \min\{\size{w_i-t},40-\size{w_i-t}\}$.  To see this, let $w' = (w_1, w_2, \dots, w_{i-1}, t, w_{i+1}, \dots, w_{k})$; it is clear that some shortest path from $s_{i,t}$ to $w$ contains $w'$, so $d(s_{i,t}, w) = d(s_{i,t},w') + d(w',w) = 40 + \min\{\size{w_i-t},40-\size{w_i-t}\}$.  In particular, $d(s_{i,t},w)$ depends only on the $i$th coordinate of $w$.

\begin{theorem}\label{thm:degen_tight}
For all positive integers $k$, the graph $G_k$ has degeneracy $3^k-1$ and localization number $k$.
\end{theorem}
\begin{proof}
The $k$-fold strong product of copies of $C_{40}$ is regular of degree $3^k-1$, so clearly the degeneracy of $G_k$ is at least $3^k-1$.  By Theorem~\ref{thm:degeneracy}, we now have $\loc{G_k} \ge \log_3(3^k) = k$.  To complete the proof, it suffices to show that $k$ cops can locate a robber on $G$ and hence $\loc{G_k} \le k$.

Label the cops $1, 2, \dots, k$.  
%Throughout the game, cop $i$ will only probe the satellites $s_{i,0}$ and $s_{i,10}$.  
Before presenting the full details of the cops' strategy, we give an overview.  In general, on each turn of the game, the robber either occupies some core vertex $(z_1, z_2, \dots, z_k)$ or some vertex on a thread ending at some such core vertex.  (It is also possible that the robber could occupy a satellite, but this case will be very easily dispatched.) To locate the robber, the cops need to determine coordinates $z_1, \dots, z_k$.  For each $i \in \{1, \dots, k\}$, cop $i$ will attempt to determine $z_i$, which she does by probing either $s_{i,0}$ or $s_{i,10}$. As we will show, it is relatively easy for the cops to locate the robber provided that he begins in the core and never leaves, and it likewise easy for the cops to locate the robber provided that they can be certain he has left the core; the key difficulty is in distinguishing between these two cases.
%As we will see, determining $z_i$ is fairly straightforward provided that the robber begins in the core and never leaves.  Likewise, it is not difficult for the cops to locate the robber provided that they can be certain he has left the core.  The key difficulty will be in distinguishing between these two cases.  

We present the cops' strategy in three stages.  Before presenting the cops' main strategy, we explain how they can locate the robber if at some point in the game some cop observes a distance smaller than 40 or larger than 60 (which would immediately indicate that the robber has left the core).  Next, we give the cops' main strategy, and we explain how this enables them to locate the robber provided that they can be certain he has never left the core.  Finally, we explain how the cops proceed if there is some ambiguity as to whether or not the robber has ever left the core.

First suppose that at some point in the game, some cop $c$ observes a distance strictly less than $40$; letting $v_c$ denote the satellite that this cop has just probed, the cops can infer that the robber occupies some thread emanating from $v_c$.  Let $z$ be the core vertex at the other end of this thread, and let $z = (z_1, z_2, \dots, z_{k})$.  The cops seek to determine the coordinates of $z$, which they can do with their next probe.

Say that cop $c$, when probing $v_c$, observed a distance of $40-d$ for some positive integer $d$.  If $d=40$, then the robber occupies $v_c$ and the game is over, so suppose otherwise.  Cop $c$ has already determined $z_c$: it is $0$ if $v_c = s_{c,0}$ and 10 if $v_c = s_{c,10}$.  Likewise, she knows the robber's distance from $z$ along the thread.  At the time of the cops' first probe, the robber was on an internal vertex in some thread, so with his ensuing move, he can only have moved along the thread.  With her next probe, cop $c$ probes $v_c$ again, and again she learns the robber's distance from $z$ along the thread. Once again we may suppose that the robber does not occupy $v_c$, since otherwise he has been located.  %On the next turn, she probes $v_i$ again, and again she learns the robber's distance from $z$ along the thread.

Now consider some other cop $i$.  Cop $i$ can determine the distance from her first probe to $z$ by taking the distance she just observed and subtracting $d$, since the shortest path from her probe to the robber must pass through $z$, and the robber is $d$ steps from $z$ along the thread.  On her next turn, she probes whichever of $s_{i,0}$ and $s_{i,10}$ she did not just probe.  As before, she can determine her distance to $z$ using the results of cop $c$'s second probe.  At the time of the cops' first probe, the robber was on an internal vertex in some thread, so with his ensuing move, he can only have moved along the thread.   Thus, the coordinates of the endpoint of that thread -- that is, $z$ -- cannot have changed with his last move.  Cop $i$ knows, from her two probes, both $\min\{z_i,40-z_i\}$ and $\min\{\size{z_i-10},40-\size{z_i-10}\}$; using this information, she can uniquely determine $z_i$.  Collectively, the cops can uniquely determine $z$, so they know which thread the robber occupies; since they also know the robber's distance from $z$ along the thread, they have successfully located him.

%Now suppose instead that at some point, some cop $c$ observed a distance of $60+d$ for some positive integer $d$; as before, let $v_c$ denote the satellite that she has just probed.  
Now suppose instead that at some point, some cop observed a distance of $60+d$ for some positive integer $d$.  Once again this indicates that the robber occupies some thread, but this time the cops cannot necessarily determine which satellite that thread emanates from.  If any cop observed a distance smaller than 40, then the cops can locate the robber using the strategy above, so suppose otherwise.  On the cops' next turn, each cop $i$ probes whichever of $s_{i,0}$ and $s_{i,10}$ she did not just probe.  If the robber still occupies a vertex internal to the thread, then some cop must observe a distance smaller than 40, and once again the cops can locate the robber.  Otherwise, the cops know that the robber has just moved into the core; hence, at the time of the cops' first probe, the robber was exactly one step from the core.  Taking this into account, each cop $i$ now has enough information to determine $z_i$ as in the previous paragraph, so once again the cops can locate the robber.
%%Otherwise, the robber must have just moved from his thread back onto the core.  Let $z$ be the robber's current position.  Similarly to the previous paragraph, each cop can determine one coordinate of $z$.  The cops know that before his last move, the robber must have been on some thread, one step away from the core.  Thus each cop $j$ can determine the distance from her initial probe to $z$ by simply subtracting 1 from the observed value, and she already knows the distance from her second probe to $z$.  Hence, she knows both $\size{z_j}$ and $\size{z_j-10}$ (modulo 20), so she can determine the $j$th coordinate of $z$.  Since the cops can determine all $k$ coordinates of $z$, they have located the robber.

We now give the cops' ``main'' strategy. %, which enables them to locate the robber provided that he never leaves the core.  
If at any point any cop observes a distance smaller than 40 or greater than 60, then the cops can locate the robber as explained above, so we assume throughout that this never happens.  The cops will attempt to determine the robber's location within three rounds.  The cops initially operate under the presumption that the robber always remains within the core, but they will remain alert for any indications that this may not be the case.  Under this presumption, let $x = (x_1, x_2, \dots, x_{k})$ denote the robber's position at the time of the cops' first probe, let $x' = (x'_1, x'_2, \dots, x'_{k})$ denote his position at the time of the second probe, and let $x'' = (x''_1, x''_2, \dots, x''_{k})$ denote his position at the time of the third probe.  The cops aim to determine $x''$ and thus win the game with their third probe.  %If the cops can be certain that the robber remains within the core at all times, then determining $x''$ is straightforward, but determining whether or not the robber has left the core could be difficult. 

Below we describe a strategy for each individual cop. For each $i \in \{1, \dots, k\}$, cop $i$ aims to determine $x''_i$. Depending on the results of her probes, she may detect the possibility that the robber might have entered the interior of a thread emanating from either $s_{i,0}$ or $s_{i,10}$; if this happens, then we say that coordinate $i$ is \emph{critical}.  Should any coordinates be deemed critical within the cops' first three turns, the cops will need additional probes to determine whether or not the robber has, in fact, left the core.

On the cops' first turn, each cop $i$ probes satellite $s_{i,0}$; suppose she observes a distance of $40+d_i$ for some nonnegative integer $d_i$.  We consider five possibilities based on the value of $d_i$:

\begin{itemize}
\item[(a)] $2 \le d_i \le 8$.  In this case, either $2 \le x_i \le 8$ or $32 \le x_i \le 38$; consequently, either $0 \le x''_i \le 10$ or $30 \le x''_i \le 39$.  On her second and third turns, cop $i$ probes $s_{i,10}$.  She can now uniquely determine $x''_i$, as all 21 possible values for $x''_i$ yield different distances from $s_{i,10}$.
%\item {\em Case (2)}: $d_i \in \{8,12\}$.  In this case, either $x_i \in \{-12,-8\}$ or $x_i \in \{8,12\}$.  On her second turn, the cop probes $v_{i,10}$; suppose the robber is at distance $40+d'_i$.  If $1 \le d'_i \le 3$, then the cop can infer $x'_i \in \{7,8,9,11,12,13\}$; otherwise, she can infer $x'_i \in \{-13,-12,-11,-9,-8,-7\}$.  In either case, she can uniquely determine $x''_i$ on her third turn by probing $v_{i,0}$.
\medskip
\item[(b)] $12 \le d_i \le 20$.  In this case, $12 \le x_i \le 28$ so $10 \le x''_i \le 30$.  As in Case (1), by probing $s_{i,10}$ on her next two turns, cop $i$ can uniquely determine $x''_i$.
\medskip
\item[(c)] $d_i=1$.  In this case, $x_i \in \{39,1\}$.  On her second turn, cop $i$ probes $s_{i,10}$; say she observes a distance of $40+d'_i$.  If $d'_i \not = 10$, then she can determine $x''_i$ by probing $s_{i,10}$ on her third turn.  If instead $d'_i = 10$, then more care is needed.  We know that $x'_i = 0$.  This is problematic, since between the cops' second and third probes, the robber could leave the core and enter the interior of a thread emanating from $s_{i,0}$.  Regardless, on her third turn, cop $i$ probes $s_{i,10}$.  If she observes a distance of 49 then she knows that $x''_i = 1$, and if she observes a distance of 50 then she knows that $x''_i = 0$.  If she observes a distance of $51$, then either $x''_i = 39$ or the robber has entered the interior of a thread emanating from $s_{i,0}$, but she cannot determine which; in this case, we deem coordinate $i$ to be critical.
\medskip
\item[(d)] $d_i=0$.  Here, we know $x_i = 0$.  Again, this indicates that the robber might leave the core and enter a thread emanating from $s_{i,0}$.  On her second turn, cop $i$ probes $s_{i,0}$ once again; assuming that she doesn't observe a distance smaller than 40, we must have $x'_i \in \{39,0,1\}$.  On her third turn, she probes $s_{i,10}$.  As in Case (3), if she observes any distance other than 51 then she can determine $x''_i$; otherwise, she knows that either $x''_i = 39$ or the robber has entered the interior of a thread emanating from $s_{i,0}$, and again coordinate $i$ is critical.
\medskip
\item[(e)] $9 \le d_i \le 11$.  On her second turn, cop $i$ probes $s_{i,10}$; assuming that she does not observe a distance smaller than 40, she can verify that the robber has not yet left the core.  On her third turn, she probes $s_{i,0}$.  As in Cases (4) and (5), she may be able to conclude that the robber has not left the core, in which case she can determine $x''_i$.  Otherwise, she knows only that either the robber has entered the interior of some thread emanating from $s_{i,10}$ or $x''_i = 11$; in this case, once again coordinate $i$ is critical.
\end{itemize}

After the cops' third probe, if there are no critical coordinates, then the cops can be certain that the robber hasn't left the core, and thus (as outlined above) they can uniquely determine his position.  Suppose instead that at least one coordinate is critical.  For each $i \in \{1, \dots, k\}$, let $y_i$ denote cop $i$'s ``predicted'' value for $x''_i$ -- that is, the value of $x''_i$ provided that the robber has not left the core.  After the cops' third probe and the robber's ensuing turn, let $(z_1, z_2, \dots, z_{k})$ denote either the robber's current position (if in fact he remains in the core) or the core vertex at the end of the thread on which the robber resides (if he has left the core).  The cops play as follows, with each cop $i$'s strategy depending on the value of $y_i$.
\begin{enumerate}
\item[(a)] If $y_i = 39$, then cop $i$ probes $s_{i,0}$.  If she observes a distance smaller than 40, then the cops can locate the robber as explained earlier. If she observes a distance of exactly 40, then the robber must be in the core with $z_i = 0$.  If she observes a distance of 41, then the robber cannot possibly have just left the interior of a thread emanating from $s_{i,0}$, so $x''_i = y_i = 39$.  Consequently, the robber must be in the core and so $z_i = 39$, since if the robber had just entered the interior of a thread emanating from some other satellite, then cop $i$ would have observed a distance of 42.  Finally, if she observes a distance of 42, then perhaps the robber was in the core, has just entered the interior of a thread, and $z_i = 39$, or perhaps the robber remains in the core and $z_i = 38$; in this case, coordinate $i$ remains critical after the cops' turn.  

Note that the cops can uniquely determine $z_i$ provided that they can, collectively, determine whether or not the robber is currently in the core.
% if she sees 40, she knows the robber's coordinate, but he might have just left a thread.  If she sees 41, then she knows his coordinate and she knows that he can't have just entered a thread.  If she sees 42, then he might have just entered a thread -- and his coord might be -1 or -2.
\medskip
\item[(b)] If $y_i = 11$, then cop $i$ probes $s_{i,10}$.  As usual, if she observes a distance smaller than 40, then the cops can locate the robber.  If she observes a distance of 40, then the robber is presently in the core and $z_i = 10$.  If she observes a distance of 41, then necessarily $z_i = 11$ and the robber remains in the core.  If she observes a distance of 42, then perhaps the robber was in the core, has just entered some thread, and $z_i = 11$, or perhaps he remains in the core and $z_i = 12$; in this last case, coordinate $i$ remains critical.
\medskip
\item[(c)] If $1 \le y_i \le 9$, then cop $i$ probes $s_{i,0}$.  Suppose she observes a distance of $40+d$ for some nonnegative integer $d$.  She now knows that either the robber remains in the core and $z_i = d$ or that the robber has entered some thread and $z_i = d-1$.
\medskip
\item[(d)] If $12 \le y_i \le 29$, then cop $i$ probes $s_{i,10}$.  As in the previous case, she can determine $z_i$ provided that the cops can deduce whether or not the robber remains in the core.
\medskip
\item[(e)] If $30 \le y_i \le 38$, then by probing $s_{i,0}$, cop $i$ can again determine $z_i$ provided that the cops can deduce whether or not the robber remains in the core.
\medskip
\item[(f)] If $y_i = 0$, then cop $i$ probes $s_{i,10}$.  If she observes a distance of 51, then the robber may have just entered the interior of some thread (possibly emanating from $s_{i,0}$), or it could instead be that the robber remains in the core and $z_i = 39$; in this case, coordinate $i$ remains critical after the cops' turn.  Otherwise, as before, the cop has enough information to determine $z_i$ provided that the cops can determine whether or not the robber remains in the core.
\medskip
\item[(g)] If $y_i = 10$, then cop $i$ probes $s_{i,0}$.  As in the previous case, if she observes a distance of 51, then the robber may have just entered the interior of some thread (possibly emanating from $s_{i,10}$), or it could instead be that $z_i = 11$; once again, coordinate $i$ remains critical after this round.  Otherwise, the cop again has enough information to determine $z_i$ provided that the cops can determine whether or not the robber remains in the core.
\end{enumerate}
In each case, if the cops can conclusively determine whether or not the robber is currently in the core, then cop $i$ can determine $z_i$ for all $i \in \{1, \dots, k\}$ and hence the cops can locate the robber.  If any cop observes a distance of exactly 40, then the robber must be in the core, so the cops can locate him.  If all distances observed exceed 40 but no coordinates are critical after this last round of probes, then again the the robber must be in the core and the cops can locate him.  Finally, suppose one or more coordinates are critical after this round, so the cops cannot tell whether or not the robber is presently in the core.  By the strategy above, the cops can be certain that the robber does not occupy the endpoint, in the core, of any thread; if he did, then they would have noticed this, concluded that he was in the core, and located him.  Thus, if in fact the robber does presently reside in the core, then he cannot possibly move into the interior of a thread with his next move.  Consequently, if the cops repeat the above strategy once more on their next turn, then there cannot be any critical coordinates; thus the cops can determine whether or not the robber is now in the core, after which they can locate him.
\end{proof}

We do not have a construction demonstrating the tightness of Theorem~\ref{thm:degeneracy_bipartite}.  However, the localization number of the hypercube $Q_{k}$ exceeds the bound in Theorem~\ref{thm:degeneracy_bipartite} by no more than 2; see Theorem~\ref{thm:hypercube}.
\end{section}

\begin{section}{Outerplanar graphs}

Bosek et al.\ \cite{nisse1} showed that $\loc{G}$ can be unbounded on the class of planar graphs and asked whether the same is true of outerplanar graphs. They answer this question in the negative in \cite{nisse2}, by showing that $\loc{G} \le 3$ when $G$ is outerplanar.  They actually prove $\zeta^*(G)\le 3,$ where $\zeta^*(G)$ is the corresponding parameter in the \emph{centroidal localization game}. In each round of this game (which is similar to the localization game), the cops receive only the relative distances between their location and the robber. More precisely, in this game, if the cops probe $u_1,u_2,\ldots , u_k$ and the robber is on $y$, then for all $1\le i <j \le k$ the cops learn whether $d(u,y)=0$, $d(u_i,y)=d(u_j,y)$, $d(u_i,y)<d(u_j,y),$ or $d(u_i,y)>d(u_j,y).$  Note that for all graphs $G$, we have that $\loc{G} \le \zeta^*(G)$.

Bosek et al.\ \cite{nisse2} ask whether there exists an outerplanar graph with localization number 3; that is, whether their bound on $\loc{G}$ is tight. We answer this question by showing that in fact $\loc{G} \le 2$ when $G$ is outerplanar.  (This bound is clearly tight; for example, $\loc{C_3} = 2$.)

Recall that a {\em block} of a graph $G$ is a maximal 2-connected subgraph of $G$; every graph is the edge-disjoint union of its blocks.

\begin{theorem}\label{thm:outerplanar}
If $G$ is an outerplanar graph, then $\loc{G} \le 2$.
\end{theorem}
\begin{proof}
We give a strategy for two cops to locate a robber on $G$. Throughout the game, the cops will maintain a set of vertices called the {\em cop territory}.  The cop territory will be a connected subgraph of $G$, and the cops will distinguish two distinct vertices of the cop territory as the {\em endpoints} of the territory.  The cops will maintain three invariants:
\begin{description}
\item[(1)] Immediately after a probe, the cops can be certain that the robber does not occupy any vertex of the cop territory.
\item[(2)] No vertex in the cop territory, with the possible exception of the endpoints, is adjacent to any vertex outside the cop territory.
\item[(3)] Both endpoints belong to the same block of $G$.
\end{description}
We give a strategy for the cops to gradually enlarge the cop territory; since $G$ is finite, this process cannot continue indefinitely, so the cops must eventually  locate the robber.  Throughout the game, if either cop observes a distance of 0 on her probe, then she has located the robber and the cops have won; thus, in the proof below, we implicitly assume that this has not happened.

The cops' general approach is as follows.  The cops will focus on one block of $G$ at a time.  Over the course of several turns, they will ensure that the robber does not occupy any vertex of this block and, in the process, expand the cop territory to contain all vertices in the block.  They will then move on to a new block that is ``closer'' to the robber and repeat the process until they have located the robber.  Throughout the proof, $B$ will denote the block that the cops are currently probing, and $\vl$ and $\vr$ will denote the endpoints of the cop territory.  We sometimes refer to $\vl$ (respectively $\vr$) as the {\em left endpoint} (resp. {\em right endpoint}) of the cop territory, and we refer to the cop who has most recently probed $\vl$ (resp. $\vr$) as the {\em left cop} (resp. {\em right cop}).  For a vertex $v$ in $B$, we define $G_v$ to be the (possibly empty) subgraph of $G-v$ not containing any vertices of $B$.  Informally, $G_v$ is the collection of blocks ``attached to'' $v$; that is, those blocks on the other side of $v$ from $B$.  (See Figure~\ref{fig:outerplanar_gv}.)  In what follows, we will repeatedly use the following observation: for any two distinct vertices $u$ and $v$ in $B$, if the robber occupies $G_v$, then he must be closer to $v$ than to $u$.

\begin{center}
\begin{figure}[ht]
\includegraphics{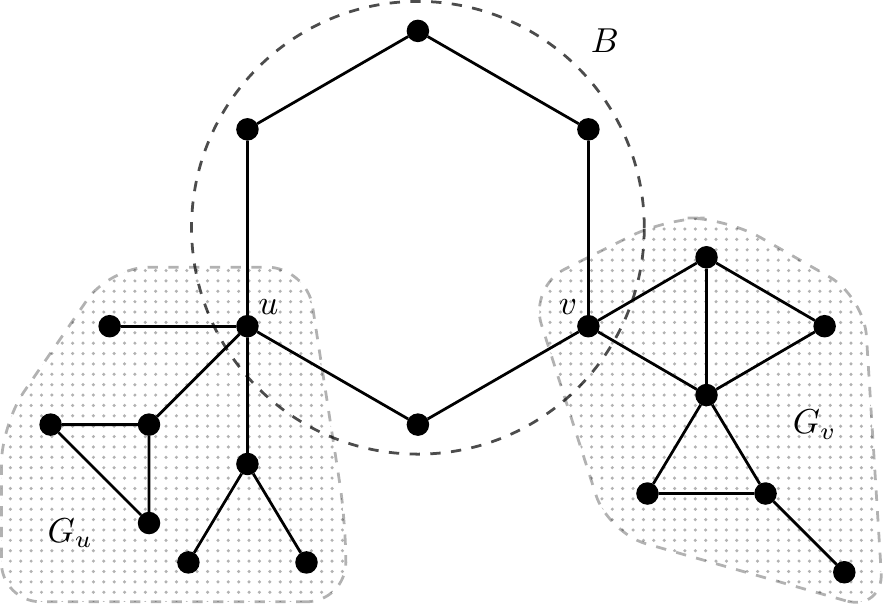}
\caption{. An outerplanar graph $G$ with subgraphs $G_u$ and $G_v$, the collections of blocks attached to $u$ and to $v$, respectively.}
\label{fig:outerplanar_gv}
\end{figure}

\end{center}

Initially, the cops choose any block $B$ of $G$, choose adjacent vertices within $B$ to comprise the cop territory, designate these vertices $\vl$ and $\vr$, and probe them.  It is evident that all three invariants hold.  To show how the cops can enlarge the cop territory, we consider the structure of $B$. 

Suppose first that $B$ is $K_2$.  Since both $\vl$ and $\vr$ are cutvertices (or pendant vertices) in $G$, the robber must be closer to one than to the other; without loss of generality, suppose he is closer to $\vl$.  The cops now know that the robber cannot be in $G_{\vr}$, so they may add all vertices of $G_{\vr}$ to the cop territory.   On their next turn, the cops choose any neighbor of $\vl$ that is not in the cop territory and designate this vertex to be the new $\vr$.  They then probe $\vl$ and $\vr$, and they add $\vr$ to the cop territory.  The robber cannot occupy $\vl$ or $\vr$ (since otherwise the cops would have located him), and he was unable to pass through $\vl$ with his previous move, so he cannot be in the cop territory.  Moreover, it is clear that no vertex in the cop territory aside from the endpoints can have any neighbor outside the cop territory.  Finally, both endpoints clearly belong to the same block of $G$ (which the cops now take as the new block $B$).  Thus all three invariants have been maintained, and the cops have successfully enlarged the cop territory.

Suppose instead that $B$ is not $K_2$.  In this case, $B$ must itself be a 2-connected outerplanar graph.  Recall that a $2$-connected outerplanar graph can be represented as a Hamiltonian cycle with non-crossing chords drawn inside it.  Consider some such representation of $B$, and label its vertices $v_1, v_2, \dots, v_n$ in clockwise cycle order.  (For convenience, we may wish to refer to $v_{n+1}$, $v_{n+2}$, etc. later in the proof; indices should be adjusted modulo $n$ where needed.)  The intersection of the cop territory with $V(B)$ will consist of vertices $v_{\ell}, v_{\ell+1}, \dots, v_{r}$ for some $\ell$ and $r$; that is, it is an ``arc'' of the outer cycle.  By symmetry, we may suppose at all times that $\vl = v_{\ell}$ and $\vr = v_r$. (Note that this means that whenever the endpoints of the cop territory change, the values of $\ell$ and $r$ change accordingly.)  Henceforth, the cops play as follows.  The left cop probes $\vl$, while the right cop probes $\vr$.  Suppose that the robber was at distance $\dl$ from $\vl$ and distance $\dr$ from $\vr$.\\

\noindent {\bf Case 1}: All vertices of $B$ belong to the cop territory.\\

If in fact all of $V(G)$ belongs to the cop territory, then the cops have won, so suppose otherwise.  By invariant (2), every vertex outside the cop territory that is adjacent to a vertex inside the cop territory must be adjacent to $\vl$ or $\vr$, so the robber must reside in either $G_{\vl}$ or $G_{\vr}$.  Since all vertices of $B$ belong to the cop territory, $\vl$ and $\vr$ are either equal or adjacent along the outer cycle of $B$. If $\vl = \vr$, then $G_{\vl} = G_{\vr}$.  If instead $\vl$ is adjacent to $\vr$, then we cannot have $\dl = \dr$; if $\dl < \dr$ then the robber occupies a vertex in $G_{\vl}$, and if $\dr < \dl$ then the robber occupies a vertex in $G_{\vr}$. We assume henceforth that the robber occupies a vertex in $G_{\vl}$; a symmetric argument suffices for the case where he occupies a vertex in $G_{\vr}$.  If $G_{\vl} \not = G_{\vr}$, then the cops add all vertices of $G_{\vr}$ to the cop territory.  To proceed, the cops must determine which component of $G_{\vl}$ contains the robber.

Within $G_{\vl}$, let $B_1, B_2, \dots, B_m$ be the blocks containing $\vl$.  For $i \in \{1, \dots, m\}$, let $C_i$ be the subgraph of $G_{\vl}$ induced by $\vl$ and all vertices in the same component of $G_{\vl}-\vl$ as the vertices of $B_i-\vl$. (Informally, $C_i$ consists of all vertices ``on the same side of'' $\vl$ as $B_i$.) Note that any two $C_i$ share only one vertex, namely $\vl$, and the $C_i$ together contain all vertices in $G_{\vl}$.  The cops aim to determine which of these components the robber occupies.  They begin by determining whether or not the robber occupies $C_1$.  If $B_1 = K_2$, then they can easily do this by probing both $\vl$ and the other vertex of $B_1$, so suppose otherwise. Within $B_1$, let $w_1, w_2, \dots, w_k$ be the neighbors of $\vl$, in clockwise order around the outer cycle of $B_1$.  The cops probe $\vl$ and $w_1$; let $\dl$ and $d_1$ denote the robber's distances from $\vl$ and $w_1$, respectively.  Note that $d_1 \in \{\dl-1,\dl,\dl+1\}$.  If $d_1 \le \dl$, then the robber must be in $C_1$.  The cops now take $B_1$ as the new block $B$, take $\vl$ and $w_1$ as the new left and right endpoints of the cop territory, and add all vertices of $C_2 \cup C_3 \cup \dots \cup C_m$ to the cop territory.

Suppose instead that $d_1 = \dl+1$.  On their next turn, the cops probe $\vl$ and $w_2$; let $\dl'$ and $d_2$, respectively, be the distances observed.  Once again, if $d_2 \le \dl'$, then the robber must be in $C_1$ and the cops play as outlined in the preceding paragraph.  Otherwise, we must have $d_2 = \dl'+1$.  We claim that for all vertices $u$ in $B_1$ that lie on the clockwise arc from $w_1$ to $w_2$ (inclusive), the robber cannot occupy either $u$ or $G_{u}$.  Suppose otherwise, let $z$ denote the robber's current position, and let $y$ denote the robber's previous position (that is, his position at the time of the cops' previous probe).  Since $d_2 = \dl'+1$, some shortest path from $w_2$ to $z$ passes through $\vl$, and thus through $w_1$ as well (since $u$ lies on the arc from $w_1$ to $w_2$). Consequently, we have that $d(w_1,z) = d(w_2,z)-2 = d_2-2 = \dl'-1$.  Because $y$ and $z$ are adjacent, we also have $d_1 = d(w_1,y) \le d(w_1,z)+1 = \dl'$.  Similarly, $d(w_2,y) = d_1-2 = \dl-1$, hence $d_2 = d(w_2,z) \le d(w_2,y)+1 = \dl$.  Thus $\dl \ge d_2 = \dl'+1$, and yet $\dl' \ge d_1 = \dl+1$, so $\dl \ge \dl'+1 \ge \dl+2$, a contradiction.

The cops next probe $\vl$ and $w_3$, use this information to determine whether or not the robber lies between $w_2$ and $w_3$, and proceed in this manner until they either determine that the robber occupies $C_1$ (at which point they proceed as explained earlier) or exhaust all neighbors of $\vl$ in $B_1$.  In the latter case, they repeat the process in $B_2$, then $B_3$, and so forth.  Since the cops probe $\vl$ on every turn, the robber cannot move between the $C_i$, so eventually the cops determine which $C_i$ contains the robber, at which point they enlarge the cop territory and proceed into a new block.\\

\noindent {\bf Case 2}: $\dl = 1, \dr = 1,$ or both.\\

If both $\dl$ and $\dr$ are 1, then the robber's position is uniquely determined, since $\vl$ and $\vr$ can have at most one common neighbor outside the cop territory.  Thus, suppose that $\dl = 1$ but $\dr > 1$; a symmetric argument suffices when $\dr = 1$ and $\dl > 1$.  Note that since $\dr > \dl$, the robber cannot occupy $G_{\vr}$; if any vertices of $G_{\vr}$ do not yet belong to the cop territory, then the cops add them.  We consider two cases.  (Refer to Figure~\ref{fig:outerplanar_case2}.)
\begin{itemize}
\item[(a)] Suppose $\vl$ is adjacent to $v_{r+1}$.  Since $\dr > 1$, the robber cannot enter $\vr$ on his ensuing turn.  The cops now add $v_{r+1}$ to the cop territory and take $\vl$ and $v_{r+1}$ as the new endpoints.  Due to the presence of edge $\vl v_{r+1}$, there cannot be any edges joining $v_r$ to vertices of $B$ not in the cop territory, so invariant (2) still holds.  The cops have successfully enlarged the cop territory.
\medskip
\item[(b)] Suppose $\vl$ is not adjacent to $v_{r+1}$.  Of all the neighbors of $\vl$ in $B$ that are outside the cop territory, let $v_{s}$ denote the one furthest counterclockwise.  On their next turn, the left cop probes $\vl$ while the right cop probes $v_{s-1}$.  The cops now take $\vl$ and $v_{s-1}$ to be the left and right endpoints of the cop territory, respectively, and add to the cop territory $v_{r+1}, \dots, v_{s-1}$ along with $G_{v_{r+1}}, \dots, G_{v_{s-1}}$.  The robber cannot possibly occupy the cop territory: by choice of $s$ and the fact that $\dl = 1$, prior to his last move the robber could not have occupied $v_i$ or $G_{v_i}$ for any $i \in \{r+1, \dots, s-1\}$, and he cannot have reached any of these in just one step -- except perhaps for $v_{s-1}$, which the cops have just probed.  Thus invariant (1) holds; invariants (2) and (3) clearly hold as well.  Finally, since $\vl$ is not adjacent to $v_{r+1}$, we have $s \ge r+2$.  Thus $v_{s-1}$ is further clockwise than $\vr$, so the cops have enlarged the cop territory.\\
\end{itemize}

\begin{figure}[ht]
\begin{center}
\includegraphics{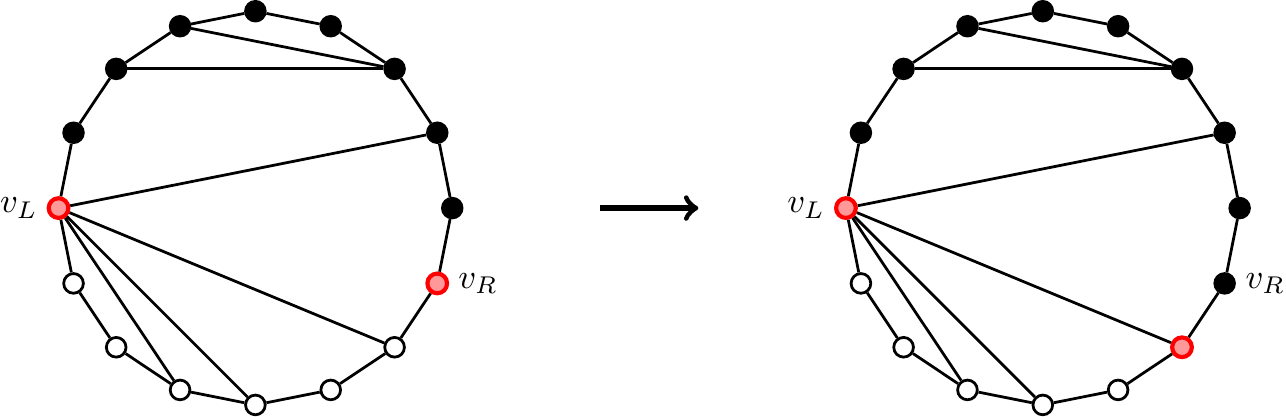}

\vspace{20.mm}

\includegraphics{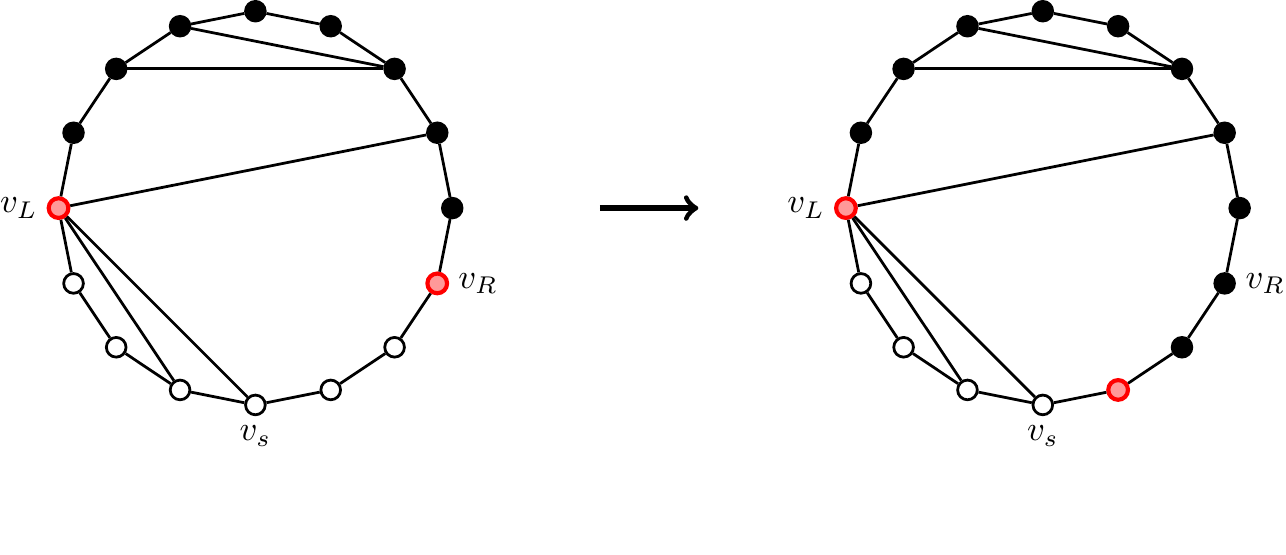}

\end{center}
\caption{. Top: Case 2(a).  Bottom: Case 2(b).  Filled vertices represent the interior of the cop territory; shaded vertices represent the endpoints; unfilled vertices represent the robber territory.  Only block $B$ is pictured.}
\label{fig:outerplanar_case2}
\end{figure}

\par
\noindent {\bf Case 3}: $\dl > 1, \dr > 1,$ and exactly one of $\vl$ and $\vr$ lies on a chord of $B$ joining it to a vertex outside the cop territory.\\

Suppose that $\vl$ lies on such a chord while $\vr$ does not; the other case is similar.
\begin{itemize}
\item[(a)] If all vertices of $G_{\vr}$ belong to the cop territory, then on their next turn the cops add $v_{r+1}$ to the cop territory as the new right endpoint.  (Note that since $\dr > 1$, the robber could not have entered $\vr$ on his last turn, so he cannot be in the cop territory.)
\medskip
\item[(b)] If part of $G_{\vr}$ does not belong to the cop territory and $\dr \ge \dl$, then the robber cannot occupy $G_{\vr}$, so the cops may safely add all vertices of $G_{\vr}$ to the cop territory.
\medskip
\item[(c)] Suppose part of $G_{\vr}$ does not belong to the cop territory and $\dr < \dl$.  If any vertices of $G_{\vl}$ do not yet belong to the cop territory, then the cops add them now.  Out of all neighbors of $\vl$ in $B$ that do not belong to the cop territory, let $v_s$ be the one furthest counterclockwise.  We claim that for all $i \in \{\ell-1, \ell-2, \dots, s\}$, the robber cannot have occupied either $v_i$ or $G_{v_i}$ immediately after the cops' probe.  To see this, note that the shortest path from $\vr$ to any such vertex must pass through $\vl$ or $v_s$, and $\vl$ is at least as close to both of these vertices as $\vr$.  Thus on their next turn the cops may take $\vr$ and $v_s$ as the new endpoints of the cop territory and add $v_i$ and all vertices of $G_{v_i}$ for all $i \in \{\ell-1, \ell-2, \dots, s\}$.\\
\end{itemize}

\noindent {\bf Case 4}: $\dl > 1, \dr > 1,$ and both $\vl$ and $\vr$ lie on chords of $B$ joining them to vertices outside the cop territory.\\

Of all vertices of $B$ adjacent to $\vl$, let $v_{s}$ be the farthest counterclockwise; of all vertices of $B$ adjacent to $\vr$, let $v_{t}$ be the farthest clockwise.  Let $H_L$ denote the subgraph comprised of $\vl = v_{\ell}, v_{\ell-1}, \dots, v_{s}$ and $G_{v_{\ell}}, G_{v_{\ell-1}}, \dots, G_{v_{s}}$.  Likewise, let $H_R$ denote the subgraph comprised of $\vr = v_r, v_{r+1}, \dots, v_{t}$ and $G_{v_{r}}, G_{v_{r+1}}, \dots, G_{v_{t}}$.  The cops would like to determine which of these subgraphs (if either) the robber presently inhabits.  We consider two subcases.
\begin{itemize}
\item [(a)] Suppose first that $v_{s} \not = v_{t}$.  If the robber is in $H_L$, then $\dl < \dr$: any path from $\vr$ to a vertex in $H_L$ must pass through either $\vl$ or $v_s$, and $\vl$ is closer than $\vr$ to both of these.  Thus, if $\dl \ge \dr$, then the robber cannot be in $H_L$, so the cops add all vertices of $H_L$ to the cop territory and take $v_{s}$ and $\vr$ as the endpoints.  Invariant (1) holds since the robber did not occupy $H_L$ before his last move and could only have entered $H_L$ through $v_{s}$; invariant (2) holds by choice of $v_{s}$.  Likewise, if $\dr > \dl$, then the cops add $H_R$ to the cop territory and take $\vl$ and $v_t$ as the endpoints.
\medskip
\begin{figure}[h]
\label{fig:outerplanar_case4}
\begin{center}
\includegraphics{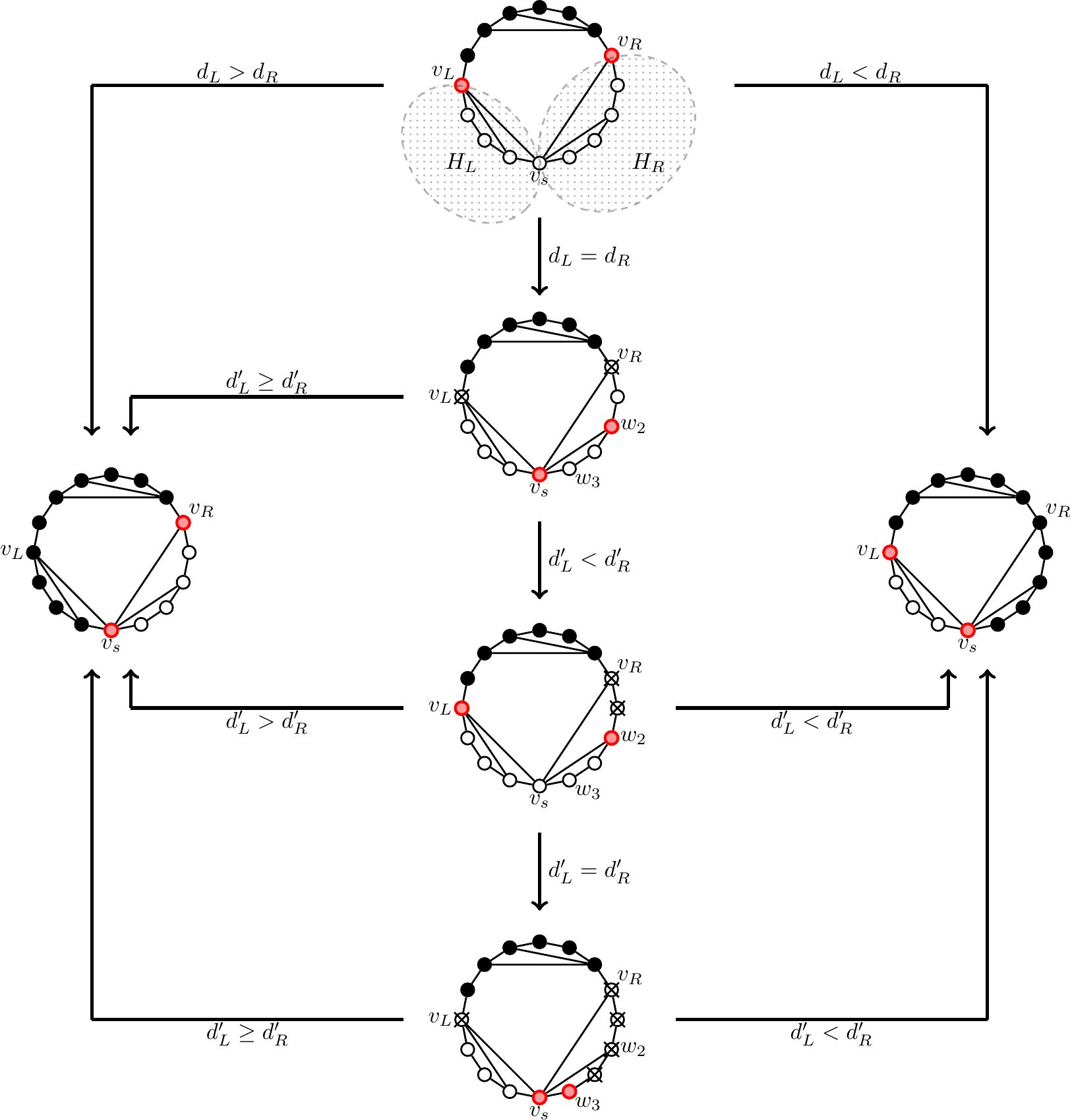}

\end{center}
\caption{. Case 4(b).  Filled vertices represent the interior of the cop territory; shaded vertices represent probes; unfilled vertices represent the robber territory.  Crossed-out vertices have been determined not to contain the robber.}
\end{figure}
\item [(b)] Suppose now that $v_{s} = v_{t}$.  This time, if the robber occupies $H_L$, we know only that $\dl \le \dr$ (and likewise if he occupies $H_R$, then $\dr \le \dl$).  If $\dl \not = \dr$, then the cops proceed as above.   Otherwise, more care is needed. In clockwise order, let $\vr = w_1, w_2, \dots, w_k$ be the neighbors of $v_{s}$ in $B$ that are counterclockwise from $v_{s}$.  For $i \in \{1, \dots, k-1\}$, let {\em sector $i$} refer to the arc of the outer cycle of $B$ from $w_i$ to $w_{i+1}$ (inclusive), together with the subgraphs $G_u$ for all vertices $u$ in this arc.  The cops aim to determine which sector (if any) the robber occupies.

On their next turn, the cops probe $v_{s}$ and $w_2$; let $d'_L$ and $d'_R$ denote the distances observed.  If $d'_L \ge d'_R$, then the robber cannot presently reside in $H_L$: every shortest path from $w_2$ to a vertex in $H_L$ must pass through either $\vl$ or $v_{s}$, and $v_{s}$ is closer to both of these than $w_2$ is.  In this case, as before, the cops may add all vertices of $H_L$ to the cop territory and take $v_{s}$ and $\vr$ as the endpoints.  Thus we may suppose that $d'_{L} < d'_R$; since $v_{s}$ and $w_2$ are adjacent, we must have $d'_R = d'_{L}+1$.

We claim that the robber cannot occupy sector 1.  Suppose to the contrary that the robber does occupy some vertex $u$ in sector 1, and note that $u \not = w_2$ (since the cops have just probed $w_2$).  Since $d'_R = d'_{L}+1$, some shortest path from $w_2$ to the robber passes through $v_{s}$ and, since the robber is in sector 1, through $\vr$ as well.  Thus, the distance from $\vr$ to $u$ is $d'_{L}-1$; since $u$ is adjacent to the robber's previous position, $\dr \le d(\vr,u)+1 = d'_{L} = d'_R-1$.  On the cops' previous turn (when they probed $\vl$ and $\vr$), we had $\dl = \dr$, so some shortest path from $\vr$ to the robber passed through $v_{s}$; since $u$ is in the interior of sector 1, the robber must have been in sector 1 on the previous turn, so this path must also have passed through $w_2$.  Thus, the distance from $w_2$ to the robber on that turn was $\dr-2$, so $d'_R = d(w_2,u) \le (\dr-2)+1 = \dr-1$.  We now have
$$\dr \le d'_R-1 \le (\dr-1)-1 = \dr-2,$$
a contradiction.

After the cops probe $v_s$ and $w_2$, and after the robber makes his ensuing move, he still cannot have entered the cop territory: since $\dl = \dr \ge 2$, he cannot have passed through either $\vl$ or $\vr$.  Moreover, before the robber's most recent move, the cops deduced that he was not in sector 1; hence he cannot have entered the interior of sector 1.  The cops now repeat this strategy, but with $w_2$ taking the place of $\vr$.  In particular, on their next turn, they probe $\vl$ and $w_2$; let $\dl$ and $\dr$ be the distances observed.  If $\dl \not = \dr$, then they can add either $H_L$ or $H_R$ to the cop territory, as before.  If $\dl = \dr = 1$, then the robber must occupy $v_{s}$.  If $\dl = \dr \ge 2$, then on their next turn the cops probe $v_{s}$ and $w_3$.  Depending on the results of that probe, the cops can either add $H_L$ to the cop territory or deduce that the robber is not in sector 2.  (Note that he also cannot be in sector 1: he cannot have traveled through $v_{s}$, and since $\dr \ge 2$, he cannot have traveled through $w_2$ either.)  Repeating this argument, the cops can eventually add either $H_L$ or $H_R$ to the cop territory and proceed.\\
\end{itemize}

\noindent {\bf Case 5}: $\dl > 1, \dr > 1,$ and neither $\vl$ nor $\vr$ lie on chords of $B$ joining them to vertices outside the cop territory.\\%, and both $v_{\ell}$ and $v_r$ are cutvertices.\\

Suppose first that both $G_{\vl}$ and $G_{\vr}$ contain vertices outside the cop territory. If $\dl \ge \dr$, then the robber cannot inhabit $G_{\vl}$, so the cops can add all vertices of $G_{\vl}$ to the cop territory.  Otherwise the robber cannot inhabit $G_{\vr}$, so the cops can instead add $G_{\vr}$ to the cop territory.   (In either case, $\vl$ and $\vr$ remain the endpoints.)\\

Finally, suppose that $G_{\vr}$ contains no vertices outside the cop territory.  (The case where $G_{\vl}$ contains no vertices outside the cop territory is similar.) Vertex $\vr$ has only one neighbor outside the cop territory, namely $v_{r+1}$.  The robber cannot have been on $v_{r+1}$ last round (since $\dr > 1$), so the cops may add $v_{r+1}$ to the cop territory and take it as the new right endpoint.

%\noindent {\bf Case 6}: $d_{\ell} > 1, d_r > 1,$ neither $v_{\ell}$ nor $v_r$ lie on chords of $B$ joining them to vertices outside the cop territory, and at least one of $v_{\ell}$ and $v_r$ is not a cutvertex.\\

%Without loss of generality, suppose $v_r$ is not a cutvertex. $v_r$ has only one neighbor outside the cop territory, namely $v_{r+1}$.  The robber cannot have been on $v_{r+1}$ last round (since $d_r > 1$), so the cops may add $v_{r+1}$ to the cop territory and take it as the new right endpoint.
\end{proof}
\end{section}

% hypercube
\begin{section}{Hypercubes}

We conclude the paper by giving an asymptotically tight upper bound on the localization number of the hypercube.

\begin{theorem}\label{thm:hypercube}
For all positive integers $n$, we have that $\loc{Q_n} \le \ceil{\log_2 n} + 2$.
\end{theorem}
\begin{proof}
We represent vertices of $Q_n$ using binary ordered $n$-tuples, where two vertices are adjacent provided that the corresponding $n$-tuples differ in exactly one coordinate.  For this proof, it will be convenient to index coordinates starting from $0$; that is, our $n$-tuples have coordinates $0$ through $n-1$ (rather than $1$ through $n$).

We show how $\ceil{\log_2 n} + 2$ cops can locate a robber on $Q_n$.  We distinguish two cops, which we refer to as ``cop $C_0$'' and ``cop $C_1$'', and we refer to the rest of the cops as {\em maintenance cops}.  The cops will locate the robber over the course of $n$ probes. Intuitively, cops $C_0$ and $C_1$ will be in charge of ``learning'' one coordinate of the robber's position in each round, while the maintenance cops will be responsible for ``updating'' any coordinates that may have changed with the robber's last move. In the first round of the game, the cops aim to determine coordinate $0$ of the robber's position.  Subsequently, for $k \in \{2, \dots, n\}$, we suppose that just before to the cops' $k$th probe they know coordinates $0$ through $k-2$ of the robber's position prior to his most recent move, and with their ensuing probe they aim to determine coordinates $0$ through $k-1$ of his current position. 

On each cop turn, $C_0$ probes the vertex $(0,0,\dots,0)$.  This probe will give the cops some insight into which ``direction'' the robber is moving. In particular, when the robber's distance to $C_0$ decreases from one round to the next, the cops know that some coordinate of the robber's position has changed from $1$ to $0$.  Likewise, if the robber's distance to $C_0$ has increased, then some coordinate of his position has changed from $0$ to $1$, and if the distance to $C_0$ remains unchanged, then the robber hasn't moved.  

For $k \in \{1, \dots, n\}$, in the $k$th round of the game, cop $C_1$ probes the vertex for which coordinate $k-1$ is 1 and all other coordinates are 0. The results of this probe, in conjunction with the results of $C_0$'s probe, allow the cops to determine coordinate $k-1$ of the robber's current position.

Finally, we explain the maintenance cops' strategy. Label these cops $0, \dots, \ceil{\log_2 n}-1$.  Fix $k \in \{1, \dots, n\}$. Recall that for $k \ge 2$, just before the cops' $k$th probe, we suppose that the cops know coordinates $0$ through $k-2$ of the robber's position prior to his last move. With this next probe, the cops aim to determine coordinates $0$ through $k-1$ of the robber's current position. We have already seen how the probes by $C_0$ and $C_1$ let the cops determine coordinate $k-1$ of the robber's position; it is the maintenance cops' job to ``update'' coordinates $0$ through $k-2$ to reflect the robber's most recent move. To do this, for each $i \in \{0, \dots, \ceil{\log_2 n}-1\}$, maintenance cop $i$ probes the vertex of $Q_n$ in which, for all $j \in \{0, \dots, n-1\}$, coordinate $j$ is 1 if and only if the binary representation of $j$ has a 1 in the ``$2^i$'' bit. (If $k=1$, then there is no need to update any coordinates of the robber's position; however, the maintenance cops still probe these vertices, since the results will be needed in the next round of the game.) 

Now suppose $k\ge 2$ and suppose that on the robber's last turn, coordinate $j$ of his position changed from a 0 to a 1. (The case where some coordinate changes from 1 to 0 is symmetric, and the probe by $C_0$ allows the cops to distinguish between these cases -- as well as to detect the case where the robber remains in place.) Those maintenance cops probing a vertex where coordinate $j$ is 1 see that the robber has moved one step closer to their probes, while the others see that he has moved one step farther away.  Thus, for each $i \in \{0, 1, \dots, \ceil{\log_2 n} - 1\}$, maintenance cop $i$ can determine whether the binary representation of $j$ has a 0 or a 1 in the $2^i$ bit.  Between them, the cops have enough information to determine $j$. Since the cops now know which coordinate of the robber's position has changed, they can update their information about coordinates 0 through $k-2$ of his position (if indeed $0 \le j \le k-2$); in total, the cops now know coordinates $0$ through $k-1$ of the robber's position, as desired.

After their $n$th probe, the cops know all $n$ coordinates of the robber's position, and so they have located him.
\end{proof}

The cop strategy used above can actually be applied to a slightly more general class of graphs. Recall that the {\em Cartesian product} of graphs $G$ and $H$, written $G \cart H,$ is the graph with vertex set $V(G) \times V(H)$, where $(u,v)$ is adjacent to $(u',v')$ provided that $u$ is adjacent to $u'$ in $G$ and $v=v'$, or $u = u'$ and $v$ is adjacent to $v'$ in $H$.
\begin{theorem}\label{thm:hypercube_generalization}
If $G = G_0 \cart G_1 \cart \dots \cart G_{n-1}$, where each $G_i$ is a path, then $\loc{G} \le \ceil{\log_2 n} + 2$.
\end{theorem}
In lieu of a full proof of Theorem~\ref{thm:hypercube_generalization}, we explain how the strategy from Theorem~\ref{thm:hypercube} can be adapted.  As before, we represent vertices of $G$ as ordered $n$-tuples, but they need no longer be binary $n$-tuples; instead, for $i \in \{0, \dots, n-1\}$, coordinate $i$ can take on any value from $0$ up to $\size{V(G_i)}-1$.

To locate a robber on $G$, the cops follow the same strategy as in Theorem~\ref{thm:hypercube}, with one change: for all $i \in \{0, \dots, n-1\}$, whenever a probe would have originally had 1 in coordinate $i$, the probe should instead have $\size{V(G_i)}-1$ in that coordinate.  (All other coordinates remain unchanged.)  As before, in the $k$th round of the game, the cops aim to determine the first $k-1$ coordinates of the robber's position.  It is straightforward to verify the following:
\begin{itemize}
\item In round $k$, cops $C_0$ and $C_1$ can determine coordinate $k-1$ of the robber's position.
\medskip
\item In each round, $C_0$ can determine whether the robber has incremented some coordinate of his position, decremented some coordinate, or remained in place.
\medskip
\item In each round, if the robber has changed his position, then the maintenance cops can determine which coordinate has changed.  As in the original strategy, when the robber increments some coordinate $j$ of his position, those maintenance cops whose probe has $\size{V(G_j)}-1$ in that coordinate will see that the robber has moved closer to them, while the rest will see that he has moved farther away; collectively, the maintenance cops have enough information to determine $j$.  (A similar argument works if the robber decrements some coordinate of his position.)\\
\medskip
\end{itemize}

Theorems~\ref{thm:degeneracy_bipartite} and \ref{thm:hypercube} together show that $\ceil{\log_2 n} \le \loc{Q_n} \le \ceil{\log_2 n} + 2$.  It is interesting to note that although the localization number and metric dimension are closely connected, we know $\loc{Q_n}$ up to an additive constant, but we know only that $\mathrm{dim}(Q_n) \sim \frac{2n}{\log_2 n}$ (see \cite{cm,er,lind}).  Thus not only do the two parameters differ by a great deal, we also have much tighter bounds on the localization number.
\end{section}

\begin{section}{Acknowledgments}
The authors are grateful to the anonymous referees, whose comments and suggestions greatly helped improve the presentation of this paper.
\end{section}

\end{document}